\documentclass[a4paper,12pt]{amsart}  
\usepackage{amsmath,amssymb,amsthm}     
\usepackage{enumitem}

\theoremstyle{plain}      
 
\newtheorem{thm}{Theorem}[section]     
\newtheorem{theorem}[thm]{Theorem}     
     
\newtheorem{corollary}[thm]{Corollary}     
     
\newtheorem{lemma}[thm]{Lemma}     
     
\newtheorem{proposition}[thm]{Proposition}

\theoremstyle{definition}      
     
\newtheorem{definition}[thm]{Definition}  
\newtheorem{example}[thm]{Example} 
\newtheorem{remark}[thm]{Remark}

\DeclareMathAlphabet{\doba}{U}{msb}{m}{n}         
\gdef\mC{\doba{C}}

\gdef\mN{\doba{N}}

\gdef\mR{\doba{R}}

\gdef\mZ{\doba{Z}}

\def\cC{\mathcal{C}}

\def\cL{\mathcal{L}}

\def\di{{\rm d}}

\let\<\langle 
\let\>\rangle 
\let\ti\tilde
\let\witi\widetilde

\newcommand{\definedas}{\mathrel{\raise.095ex\hbox{\rm :}\mkern-5.2mu=}}

\let\ol\overline

\title{On toric locally conformally K\"ahler manifolds}

\author{Farid Madani, Andrei Moroianu, Mihaela Pilca}

\address{Farid Madani\\Institut f\"ur Mathematik\\
Goethe Universit\"at Frankfurt\\Robert-Mayer-Str. 10, 60325 Frankfurt am Main, Germany}
\email{madani@math.uni-frankfurt.de}

\address{Andrei Moroianu \\ Laboratoire de Math\'ematiques de Versailles, UVSQ, CNRS, Universit\'e Paris-Saclay, 78035 Versailles, France }
\email{andrei.moroianu@math.cnrs.fr}

\address{Mihaela Pilca\\Fakult\"at f\"ur Mathematik\\
Universit\"at Regensburg\\Universit\"atsstr. 31 
D-93040 Regensburg, Germany
\emph{and} 
Institute of Mathematics ``Simion Stoilow" of the Romanian Academy, 
21, Calea Grivitei Str.
010702-Bucharest, Romania}
\email{mihaela.pilca@mathematik.uni-regensburg.de}

\subjclass[2010]{53A30, 53B35, 53C25, 53C29, 53C55}

\keywords{locally conformally K\"ahler structure,  Hopf surface, toric manifold, twisted Hamiltonian}

\begin{document}

\begin{abstract}
We study compact toric strict locally conformally K\"ahler manifolds. We show that the Kodaira dimension of the underlying complex manifold is $-\infty$ and that the only compact complex surfaces admitting toric strict locally conformally K\"ahler metrics are the diagonal Hopf surfaces. We also show that every toric Vaisman manifold has lcK rank 1 and is isomorphic to the mapping torus of an automorphism of a toric compact Sasakian manifold.
\end{abstract}
\maketitle

\section{Introduction}

When searching for the {\it best} Hermitian metrics on a compact complex manifold $(M,J)$, one is naturally led to consider K\"ahler metrics. However, there are well-known topological obstructions which severely restrict the class of compact complex manifolds carrying such metrics. 
A more general class of compatible metrics, which was introduced by I.~Vaisman in the 80's, are the {\it locally conformally K\"ahler (lcK)} metrics. These are characterized by the condition that around any point in $M$, the metric $g$ can be conformally rescaled to a K\"ahler metric. If this metric can be globally defined, the structure is globally conformally K\"ahler~(gcK), otherwise it is called {\it strict} lcK. The topological obstructions imposed by the existence of lcK metrics are less restrictive than in the K\"ahler case. In complex dimension $2$, for instance, it was widely believed that any compact complex surface admits an lcK metric, until F. Belgun \cite{b} proved that some Inoue surfaces do not admit any lcK metric.

In this paper we investigate compact strict lcK manifolds admitting an effective torus action of maximal dimension by twisted Hamiltonian biholomorphisms. Such structures are called toric lcK and were introduced in \cite{p}. 

The paper is organized as follows: after some preliminaries on lcK manifolds, we study the properties of their automorphism group in Section~\ref{autom} and give several general results about lifts of group actions on lcK manifolds in Section~\ref{group}. 

In Section \ref{toricVaisman} we use the special decomposition of the space of harmonic 1-forms on compact Vaisman manifolds to show that every toric Vaisman manifold has first Betti number $\mathrm{b}_1=1$ and is thus isomorphic to the mapping torus of an automorphism of a toric compact Sasakian manifold.

In Section~\ref{seckod} we show that the Kodaira dimension of every compact complex manifold admitting a toric strict lcK structure is $-\infty$ (see Theorem~\ref{kod}). 

Using this result and the Kodaira classification of non-K\"ahlerian compact complex surfaces, we show in Theorem~\ref{mainthm} that the only compact complex surfaces with a toric strict lcK metric are the {\it diagonal Hopf surfaces}.  Indeed, lcK metrics on these manifolds have been constructed by P.~Gauduchon and L.~Ornea, \cite{go}, and they turn out to admit toric $T^2$-actions. As a corollary, every toric strict lcK surface carries a toric Vaisman structure.

{\bf Acknowledgments.} This work was supported by the Procope Project No.~32977YJ and by the SFB~1085. We thank Nicolina Istrati for pointing out to us an error in a preliminary version of the paper and for useful suggestions.

\section{Preliminaries on lcK manifolds}

Let $(M^{2n},J)$ be a (connected) complex manifold of complex dimension $n\ge 2$. A K\"ahler structure on $(M,J)$ is a Riemannian metric $g$ compatible with $J$ (in the sense that $J$ is skew-symmetric with respect to $g$), and such that the 2-form $\Omega:=g(J\cdot,\cdot)$ is closed. 

Two K\"ahler structures $g$ and $g'$ on $(M,J)$ in the same conformal class are necessarily homothetic. Indeed, if $g'=fg$ for some positive function $f$, the associated 2-forms are related by $\Omega'=f\Omega$, thus $0=\di\Omega'=\di f\wedge\Omega+f\di\Omega=\di f\wedge\Omega$, which shows that $\di f=0$, since the wedge product with the non-degenerate form $\Omega$ is injective on 1-forms.

A locally conformally K\"ahler (lcK) structure on $(M,J)$ is a Riemannian metric $g$ which is locally conformally K\"ahler in the sense that there exists an open covering $\{U_\alpha\}_{\alpha\in \mathcal{A}}$ of $M$ and smooth maps $\varphi_\alpha\in C^\infty(U_\alpha)$ such that $(U_\alpha,J,e^{-\varphi_\alpha}g)$ is K\"ahler. This definition is clearly independent on the metric $g$ in its conformal class $[g]$, so one usually refers to $(J,[g])$ as being an lcK structure.

Since $e^{-\varphi_\alpha}g$ and $e^{-\varphi_\beta}g$ are K\"ahler metrics in the same conformal class on $U_\alpha\cap U_\beta$, we deduce from the above remark that $\varphi_\alpha-\varphi_\beta$ is locally constant, and thus $\di\varphi_\alpha=\di\varphi_\beta$ on $U_\alpha\cap U_\beta$. This shows the existence of a closed 1-form $\theta$, called the {\it Lee form} of the lcK structure $(J,g)$, such that $\theta|_{U_\alpha}=\di\varphi_\alpha$. If $\Omega:=g(J\cdot,\cdot)$ denotes like before the 2-form associated to $g$ and $J$, then by definition $e^{-\varphi_\alpha}\Omega$ is a closed form on $U_\alpha$, thus 
$$0=\di(e^{-\varphi_\alpha}\Omega)=e^{-\varphi_\alpha}(-\di\varphi_\alpha\wedge\Omega+\di\Omega)=e^{-\varphi_\alpha}(-\theta\wedge\Omega+\di\Omega),$$
showing that $\di\Omega=\theta\wedge\Omega$ everywhere on $M$. If the lcK metric is changed by a conformal factor, $g':=e^{-\varphi} g$, the corresponding Lee form satisfies $\theta'=\theta-\di\varphi$.

If the Lee form vanishes, $g$ is K\"ahler, if it is exact, $\theta=\di\varphi$, then $g$ is globally conformal to the K\"ahler metric $g'=e^{-\varphi}g$, and if it is parallel, $g$ is called {\it Vaisman}. In this paper we will always assume that the lcK structure is {\em strict}, in the sense that its Lee form is not exact. This definition clearly does not depend on the choice of the metric $g$ in its conformal class.

Let now $\widetilde M$ be the universal covering of an lcK manifold $(M,J,g,\theta)$, endowed with the pull-back lcK structure $(\tilde J,\tilde g,\tilde \theta)$. Since $\widetilde M$ is simply connected, $\ti\theta$ is exact, {\em i.e.} $\tilde\theta=\di\varphi$, and by the above considerations, the metric $g^K:=e^{-\varphi}\tilde g$ is K\"ahler. The group $\pi_1(M)$ acts on $(\widetilde M,\tilde J,g^K)$ by holomorphic homotheties. Indeed, since $\tilde\theta$ is $\pi_1(M)$-invariant, there exists a group morphism $\rho\colon\pi_1(M)\to (\mathbb{R},+)$, $\gamma\mapsto c_\gamma$, such that
\begin{equation}\label{rho}
\gamma^*\varphi=\varphi+c_\gamma, \quad \forall \gamma\in \pi_1(M).
\end{equation}
 We thus have $\gamma^*g^K=e^{-c_\gamma}g^K$, for every $\gamma\in \pi_1(M)$. The K\"ahler structure  $(\tilde J,g^K)$ is tautologically invariant by $\ker\rho$ and defines a K\"ahler structure denoted $(\hat J,g^K)$ on $\widehat M:=\widetilde M/\ker\rho$. The K\"ahler manifold $(\widehat M,\hat J, g^K)$ is called the {\em minimal covering} of $(M,J,g)$ (and it actually only depends on the conformal class $[g]$).

We thus obtain the following sequence of Galois coverings
$$\widetilde M \overset{\ker\rho}{\longrightarrow} \widehat M  \overset{\Gamma}{\longrightarrow} M,$$
   where $\Gamma:=\pi_1(M)/\ker\rho$. Since $\rho$ induces an injective group homomorphism $\Gamma\to \mathbb{R}^*_{+}$,
   it follows that the automorphism group $\Gamma$ of the minimal covering is isomorphic to a free Abelian group $\mathbb{Z}^k$, for some $k\in \mathbb{N}$. The integer $k$ is called the {\it rank} of the lcK structure. 

A useful tool in lcK geometry is the so-called {\em Gauduchon metric}. If $(M, J, [g])$ is a compact lcK manifold, the Gauduchon metric $g_0\in[g]$ is characterized by the fact that its Lee form is co-closed:  $\delta^{g_0}\theta_0=0$, and it is unique, up to a positive factor, in the given conformal class, ~\cite{g}.

The so-called {\em twisted differential} of an lcK structure $(M,J,g,\theta)$ is defined as follows: $\di^\theta\colon \Omega^*(M)\to \Omega^{*+1}(M)$, $\di^\theta\alpha:=\di\alpha-\theta\wedge\alpha$. Note that $\di\theta=0$ implies $\di^\theta\circ \di^\theta=0$. However, $\di^\theta$ does not satisfy the Leibniz rule. If $\alpha\in\Omega^k (M)$ and $\beta\in\Omega^\ell (M)$, then the following relation holds:
\[\di^\theta(\alpha\wedge\beta)=\di^\theta\alpha\wedge\beta+(-1)^k\alpha\wedge\di^\theta\beta+\theta\wedge\alpha\wedge\beta.\]
In particular, for $f\in \cC^\infty(M)$, we have
\begin{equation}\label{dtheta}
\di^\theta(f\beta)=\di f\wedge\beta+f\di^\theta\beta.
\end{equation}

\begin{lemma}\label{lem inj}
On a strict lcK manifold $(M, J, [g])$, the twisted differential acting on functions is injective, \emph{i.e.} $\ker(\di^\theta\colon \cC^{\infty}(M)\to\Omega^1(M))=\{0\}$. 
\end{lemma}

\begin{proof}
Let $f\in\ker(\di^\theta)$, so $f$ satisfies the linear first order differential equation $\di f-f\theta=0$. If $f$ does not vanish at any point of the manifold, then $\theta=\di \log |f|$, which is not possible, since the lcK manifold is assumed to be strict. 
Thus, $f$ vanishes at some point, which implies that $f\equiv 0$, by uniqueness of the solution. 
\end{proof}

\section{Automorphisms of lcK manifolds}
\label{autom}

An automorphism of an lcK manifold $(M, J, [g])$ is a conformal biholomorphism. We denote by $\mathrm{Aut}(M, J, [g])$ the group of all automorphisms and by $\mathfrak{aut}(M, J, [g])$ its Lie algebra. In this section we establish a few properties of the automorphisms of lcK manifolds, that will be used in the sequel. Let us recall the following notions:

\begin{definition}
Let $(M,J,g,\theta)$ be an lcK manifold. A vector field $X$ on $M$ is called \emph{twisted Hamiltonian} if there exists $h_X\in\mathcal{C}^\infty(M)$, such that $X\lrcorner \Omega=\di^\theta h_X$. The space of all such vector fields on $M$ is denoted by $\mathfrak{ham}^\theta(M)$. An action of a Lie group on $M$ is called \emph{twisted Hamiltonian} if all its  fundamental vector fields belong to $\mathfrak{ham}^\theta(M)$.
\end{definition}

\begin{definition}\label{toriclck}
An lcK manifold $(M^{2n},J,[g])$ equipped with an effective holomorphic and twisted Hamiltonian  action of the standard (real) $n$-dimensional torus $T^{n}$,
 is called \emph{toric} lcK. 
\end{definition}

Next we give a class of examples of toric Vaisman manifolds. For this purpose, let us recall the definition of a Sasakian manifold. 
A {\it Sasakian structure} on a Riemannian manifold $(S,g_S)$ is a complex structure $\ti J$ on $\mR\times S$, such that the cone metric $g^K:=e^{-2t}(dt^2+g_S)$ is K\"ahler with respect to $\ti J$, and for each $\lambda\in\mathbb{R}$, the homothety $(t,w)\mapsto (t+\lambda, w)$ is holomorphic.

\begin{example}
An example of a toric lcK manifold is the Hopf manifold $S^1\times S^{2n-1}$, whose natural Vaisman structure is toric, as noticed in \cite[Example 4.8]{p}. More generally, a class of toric Vaisman manifolds can be constructed
as follows. Given a compact toric Hodge manifold $N$ of complex dimension $n-1$, we consider the total space of the $S^1$-bundle $S$ corresponding to the integral cohomology class of its K\"ahler form, which carries a Sasakian metric $g_S$. Hence the product $\mR\times S$ is endowed with a K\"ahler structure  $(g^K:=e^{-2t}(dt^2+g_S), \ti J)$. The action of $\mZ$ by translations on $\mR$ and extended trivially on $S$ is holomorphic and isometric with respect to the gcK structure $(\ti J, \ti g:=e^{2t}g^K, 2\di t)$, which thus projects onto a Vaisman structure $(J,g, \theta)$ on the compact manifold $S^1 \times S$. By construction,  $N$ is the quotient $(S^1\times S)/\{\theta^\sharp, J\theta^\sharp\}$  and its toric K\"ahler structure is the projection of $(J,g)$. 
According to \cite[Theorem~5.1]{p}, this Vaisman structure on $S^1\times S$ is toric.
\end{example}

We continue by proving some preliminary results.

\begin{lemma}\label{theta0}
Let $(M, J, g,\theta)$ be an lcK manifold with fundamental $2$-form $\Omega$. 
There exists a morphism of Lie algebras $\sigma\colon \mathfrak{aut}(M, J, [g])\to \mR$, such that
for any $X\in\mathfrak{aut}(M, J, [g])$, we have
\begin{gather}
\di^\theta(X\lrcorner \Omega)=\sigma(X)\Omega,\label{gamma}\\
\mathcal{L}_X g=(\theta(X)+\sigma(X))g.\label{lieg}
\end{gather}
\end{lemma}

\begin{proof}
For $X\in\mathfrak{aut}(M, J, [g])$, there exists $f_X\in \cC^\infty(M)$, such that $ \cL_X\Omega=f_X\Omega$. By the Cartan formula, we have 
\begin{equation}\label{eq lie}
(f_X-\theta(X))\Omega=\cL_X\Omega-\theta(X)\Omega= \di (X\lrcorner \Omega)-\theta\wedge(X\lrcorner\Omega)=\di^\theta (X\lrcorner\Omega).
\end{equation} 
Since $(\di^\theta)^2=0$ and $\Omega$ is $\di^\theta$-closed, it follows by \eqref{dtheta} that $\di(f_X-\theta(X))\wedge\Omega=0$, which by the injectivity of $\Omega\wedge\cdot$ on $1$-forms implies that $ \di(f_X-\theta(X))=0$. Hence there exists $\sigma(X)\in \mR$, such that $f_X=\theta(X)+\sigma(X)$. Since $X$ is a holomorphic vector field, it follows that $\cL_Xg=(\theta(X)+\sigma(X))g$. 

It remains to show that $\sigma$ is a Lie algebra morphism, \emph{i.e.} that $\sigma([X,Y])=0$, for all $X,Y\in \mathfrak{aut}(M, J, [g])$. Since $\theta$ is closed and $\sigma(X)$ and $\sigma(Y)$ are constant functions, we obtain
\begin{equation}\label{thetaXY}
\theta([X,Y])=X(\theta(Y))-Y(\theta(X))=X(f_Y-\sigma(Y))-Y(f_X-\sigma(X))=X(f_Y)-Y(f_X).
\end{equation}
On the other hand, we compute
$$\mathcal{L}_{[X,Y]}\Omega=\mathcal{L}_X\mathcal{L}_Y\Omega-\mathcal{L}_Y\mathcal{L}_X\Omega=\mathcal{L}_X(f_Y\Omega)-\mathcal{L}_Y(f_X\Omega)=(X(f_Y)-Y(f_X))\Omega,$$
so that $f_{[X,Y]}=X(f_Y)-Y(f_X)$. By \eqref{thetaXY}, it follows that $\sigma([X,Y])=f_{[X,Y]}-\theta([X,Y])=0$.
\end{proof}

\begin{remark} \label{rem lift}
 Let $(M,J,[g])$ be an lcK manifold. On its universal covering $\widetilde M$ endowed with the K\"ahler structure $(g^K,\Omega^K)$, Lemma~\ref{theta0} reads 
\begin{gather}
\di(\widetilde X\lrcorner\Omega^K)=\sigma(X)\Omega^K,\label{gammanew}\\
\mathcal{L}_{\widetilde X} g^K=\sigma(X)g^K,\label{liegnew}
\end{gather}
for each vector field $X$ on $M$, whose lift on $\widetilde M$ is denoted by $\widetilde X$.
\end{remark}

We denote by $\mathrm{Aut}_s(M,J,[g])$ the subgroup of $\mathrm{Aut}(M,J,[g])$ consisting of automorphisms whose lifts to $\widetilde M$ are isometries with respect to the K\"ahler metric $g^K$. By \eqref{liegnew}, its Lie algebra is $\mathfrak{aut}_s(M, J, [g]):=\ker\sigma$. The elements of $\mathrm{Aut}_s(M,J,[g])$ are called {\it special} automorphisms of the lcK structure $(M,J,[g])$.

\begin{lemma}\label{inclusion}
On an lcK manifold $(M^{2n}, J, [g])$, the following inclusion holds:
\begin{equation}\label{incl}
\mathfrak{ham}^\theta(M)\cap\mathfrak{hol}(M)\subseteq \mathfrak{aut}_s(M, J, [g]),
\end{equation}
where $\mathfrak{hol}(M)$ denotes the set of holomorphic vector fields on $(M,J)$.
\end{lemma}
\begin{proof}
Let $X\in\mathfrak{ham}^\theta(M)\cap\mathfrak{hol}(M)$. Then there exists $h_X\in\mathcal{C}^\infty(M)$, such that $X\lrcorner\Omega=\di^\theta h_X$. We now compute using Cartan's formula and the fact that $\di=\di^\theta+\theta\wedge\cdot$
\begin{equation}
\begin{split}
\mathcal{L}_X\Omega&=\di^\theta(X\lrcorner \Omega)+\theta\wedge(X\lrcorner \Omega)+X\lrcorner\di^\theta\Omega+X\lrcorner(\theta\wedge\Omega)\\
&=\theta\wedge(X\lrcorner \Omega)+X\lrcorner(\theta\wedge\Omega)=\theta(X)\Omega.
\end{split}
\end{equation}
Since $X$ is a holomorphic vector field, it follows that $\mathcal{L}_{X}g=\theta(X)g$. Hence, $X\in\mathfrak{aut}(M,J,[g])$. Moreover, \eqref{lieg} implies that $X\in\ker\sigma$.
\end{proof}

It turns out that this inclusion is even an equality under some additional assumption:
\begin{lemma}\label{invincl}
Let $(M^{2n}, J, [g])$ be a compact lcK manifold. Assume that the minimal covering of $M$ coincides with its universal covering. Then $\mathfrak{ham}^\theta(M)\cap\mathfrak{hol}(M)= \mathfrak{aut}_s(M, J, [g]).$
\end{lemma}

\begin{proof}
 Let $\Omega^K=e^{-\varphi} \pi^*\Omega$ be the fundamental form of the K\"ahler structure on the universal covering of $M$, where $\pi^*\theta=\di\varphi$. Let us fix some $X\in\mathfrak{aut}_s(M, J, [g])$. By \eqref{gammanew}, the following equality holds $\di(\widetilde X\lrcorner \Omega^K)=0$, where $\widetilde X$ is the lift of $X$ to $\widetilde M$. Thus, there exists a function $h\in\mathcal{C}^\infty(\widetilde M)$, such that $\di h=\widetilde X\lrcorner \Omega^K$. 
 
The hypothesis implies that every non-trivial element in $\pi_1(M)$ acts by a strict homothety on $(\widetilde M,g^K)$. Thus $\pi_1(M)$ is commutative, since a commutator of homotheties is an isometry.
For each $\gamma\in\pi_1(M)$, we have $\gamma^*\varphi=\varphi+c_\gamma$, where $\gamma\overset{\rho}{\mapsto} c_\gamma$ is the group morphism defined in \eqref{rho}. Hence, we have $\gamma^*\Omega^K=e^{-c_{\gamma}}\Omega^K$, which implies $\gamma^*\di h=e^{-c_{\gamma}}\di h$, or equivalently $\di(\gamma^*h-e^{-c_\gamma}h)=0$. Therefore, there exists a constant $\lambda_\gamma$, such that $\gamma^*h-e^{-c_\gamma}h=\lambda_\gamma$. Let us fix some $\gamma_0\in\pi_1(M)\setminus\{\mathrm{id}\}$. By replacing $h$ with $h-\frac{\lambda_{\gamma_0}}{1-e^{-c_{\gamma_0}}}$, we may assume $\lambda_{\gamma_0}=0$.
We now prove that in fact $\lambda_\gamma=0$, for all $\gamma\in\pi_1(M)$. Comparing the following two relations
\[\gamma^*\gamma_0^* h=\gamma^*(e^{-c_{\gamma_0}}h)=e^{-c_{\gamma_0}}(e^{-c_{\gamma}}h+\lambda_{\gamma}),\]
\[\gamma_0^*\gamma^* h=\gamma_0^*(e^{-c_{\gamma}}h+\lambda_{\gamma})=e^{-c_{\gamma_0}-c_{\gamma}}h+\lambda_{\gamma},\]
 and using the commutativity of $\pi_1(M)$, we obtain $\lambda_{\gamma}(1-e^{-c_{\gamma_0}})=0$. As $c_{\gamma_0}\neq 0$, it follows that $\lambda_{\gamma}=0$, for any $\gamma\in\pi_1(M)$. Concluding, we have shown that
\[\gamma^*h=e^{-c_{\gamma}}h, \quad\forall \gamma\in\pi_1(M).\]
It follows that the function $e^{\varphi}h$ projects onto $M$. Moreover, on $\witi M$ we have
\[\di^\theta(e^{\varphi}h)=e^\varphi (h \theta+\di h)-e^\varphi h\theta=e^\varphi\widetilde X\lrcorner\Omega^K=\widetilde X\lrcorner\pi^*\Omega.\]
Thus, $X$ is a twisted Hamiltonian vector field on $M$, whose Hamiltonian function is the projection of $e^{\varphi}h$ onto $M$.
\end{proof}

\begin{proposition}\label{infisom}
Let $(M^{2n}, J, [g])$ be a compact lcK manifold. The following assertions hold:
\begin{enumerate}
\item[$(i)$] Every vector field in $\mathfrak{aut}(M,J,[g])$ is a Killing vector field for the Gauduchon metric $g_0\in[g]$.
\item[$(ii)$]  If $\theta_0$  denotes the Lee form of the Gauduchon metric, we have $\mathfrak{aut}_s(M,J,[g])\subseteq\ker\theta_0$.
\item[$(iii)$] For any $X\in\mathfrak{aut}_s(M,J,[g])$, its lift $\witi X$, resp. $\widehat X$, to the universal (resp. minimal) covering of $(M, J, [g])$ is a Killing vector field for the K\"ahler metric $g^K$.
\end{enumerate}
\end{proposition}

\begin{proof}
$(i)$ It suffices to prove that any $\psi\in\mathrm{Aut}(M^{2n}, J, [g])$ is an isometry of $g_0$. We first notice that the metric $\psi^*g_0\in [\psi^*g]$ is also a Gauduchon metric.  
By definition, $\psi$ preserves the conformal class, so $ [\psi^*g]= [g]$.
Since in any conformal class, a Gauduchon metric is unique up to a positive constant, there exists $c>0$, such that $\psi^*g_0=c g_0$. This means that $\psi$ is a homothety. However, on a compact manifold any homothety is already an isometry. Therefore, $\psi$ is an isometry for $g_0$.

$(ii)$ Follows from $(i)$ and \eqref{lieg}.

$(iii)$ Follows from \eqref{liegnew}.
\end{proof}

\begin{proposition}\label{isotropic} Let $(M^{2n}, J, [g])$ be a compact strict lcK manifold  endowed with an effective twisted Hamiltonian holomorphic $T^m$-action. Then the orbits of the action are isotropic with respect to the fundamental $2$-form $\Omega$ and $m\leq n$.
\end{proposition}

\begin{proof}
Let $X,Y$ be two fundamental vector fields associated to the given $T^m$-action on $M$. Lemma~\ref{inclusion} implies that $X,Y\in\mathfrak{aut}_s(M)$.  
We then compute using Lemma~\ref{theta0} and the fact that $[X,Y]=0$:
\begin{equation*}
\begin{split}
\di(\Omega(X,Y))&=\di(Y\lrcorner\Omega(X))=\mathcal{L}_Y(\Omega(X))-Y\lrcorner\di(\Omega(X))=(\mathcal{L}_Y\Omega)(X)-Y\lrcorner(\mathcal{L}_X\Omega-X\lrcorner \di\Omega)\\
&=\theta(Y) \Omega (X)-\theta(X)\Omega(Y)+Y\lrcorner X\lrcorner(\theta\wedge\Omega)=\Omega(X,Y)\theta.
\end{split}
\end{equation*}

This shows that $\di^\theta(\Omega(X,Y))=0$. By Lemma~\ref{lem inj}, we conclude that $\Omega(X,Y)=0$, so each $T^m$-orbit is isotropic with respect to $\Omega$ and hence has dimension at most $n$.
On the other hand, a known consequence of the principal orbit theorem (see \emph{e.g.} \cite[Theorem~2.8.5]{dk}) is that any $T^m$ acting effectively on a compact manifold,  acts freely on its principal orbit. Thus, it follows that $m\leq n$.

\end{proof}

\section{Group actions on lcK manifolds}
\label{group}

We start with a few general lemmas about lifts of group actions from the basis to the total space of a covering. Let $G$ be a connected Lie group acting smoothly on a connected manifold $M$, and let $\pi:\widehat M\to M$ be a covering with connected covering space. For every $x\in M$ we denote by $f_x:G\to M$ the map $f_x(a):=a\cdot x$.

\begin{lemma} \label{lift}
The action of $G$ on $M$ lifts to an action of $G$ on $\widehat M$ if and only if for each $x\in M$ and $\hat x\in\widehat M$ with $\pi(\hat x)=x$, the map $f_x$ lifts to a smooth map $f_{\hat x}:G\to \widehat M$ which maps $1_G$ to $\hat x$.
\end{lemma}
 \begin{proof}
One direction is clear. Conversely, we define the map $f:G\times \widehat M\to \widehat M$ by $(a,\hat x)\mapsto a\cdot \hat x:=f_{\hat x}(a)$. Since $\pi \circ f (a,\hat x)=a\cdot \pi(\hat x)$, the map $f$ is smooth. Moreover, for every $a,b\in G$ and $\hat x\in\widehat M$ we have 
$$\pi(a\cdot(b\cdot\hat x))=a\cdot\pi(b\cdot \hat x)=a\cdot(b\cdot\pi(\hat x))=(ab)\cdot\pi(\hat x)=\pi((ab)\cdot\hat x),$$
thus showing that $a\cdot(b\cdot\hat x)$ and $(ab)\cdot\hat x$ belong to the same fibre. By continuity of the map $f$, the set of points $(a,b,\hat x)$ such that $a\cdot(b\cdot\hat x)=(ab)\cdot\hat x$ is open, closed and contains $1_G\times G\times\widehat M\subset G\times G\times\widehat M$. Since $G\times G\times\widehat M$ is connected, the above relation holds identically, thus showing that $f$ is a group action.
\end{proof}

\begin{corollary} \label{product}
If $G=G_1\times G_2$ acts on $M$ and the action of $G_i$ on $M$ lifts to an action of $G_i$ on $\widehat M$ for $i=1,2$, then the  action of $G$ on $M$ lifts to an action of $G$ on $\widehat M$.
\end{corollary}
\begin{proof}
We apply the previous lemma, noting that for each $x\in M$ and $\hat x\in\widehat M$ the map $f_x$ lifts to the map $f_{\hat x}(a_1,a_2):=a_1\cdot(a_2\cdot\hat x)$.
\end{proof}

\begin{lemma} \label{lift2}
Assume in addition that $\widehat M$ is a Galois covering of $M$ whose automorphism group has no torsion and let $p:\widehat G\to G$ be a finite covering. If the induced action of $\widehat G$ on $M$ lifts to $\widehat M$, then the action of $G$ lifts to $\widehat M$.
\end{lemma}
\begin{proof}
We denote by $K$ the image of $\pi_1(\widehat M)$ in $\pi_1(M)$. The hypothesis shows that $K$ is a normal subgroup of $\pi_1(M)$, and $\pi_1(M)/K$ has no element of finite order (except the identity). 

For every $x\in M$ we consider as before the map $f_x:G\to M$, $a\mapsto a\cdot x$ and denote by $\hat f_x:=f_x\circ p:\widehat G\to M$. By the classical covering lemma, $f_x$ lifts to $\widehat M$ if and only if $(f_x)_*(\pi_1(G))\subset K$. By assumption, $\hat f_x$ lifts to $\widehat M$, so we have $(\hat f_x)_*(\pi_1(\widehat G))\subset K$, whence $(f_x)_*(p_*(\pi_1(\widehat G)))\subset K$. The map $f_*$ thus induces a group morphism $$\rho:\pi_1(G)/p_*(\pi_1(\widehat G))\to \pi_1(M)/K.$$
On the other hand, $p_*(\pi_1(\widehat G))$ has finite index in $\pi_1(G)$ since $\widehat G\to G$ is a finite covering, and $\pi_1(M)/K$ has no torsion, so the morphism $\rho$ vanishes identically. This shows that $(f_x)_*(\pi_1(G))\subset K$, so $f_x$ lifts to $\widehat M$. We conclude by Lemma \ref{lift}.
\end{proof}

 \begin{proposition} \label{prop lcklift}
 Let $(M^{2n},J)$ be a compact complex manifold carrying an lcK structure $(J,g,\theta)$ 
with minimal covering $\widehat M$. Let $G$ be a compact connected Lie group acting holomorphically on $M$.
  Then the following assertions hold:
  \begin{enumerate}
   \item[$(i)$] If the action of $G$ lifts to $\widehat M$, then there exists an lcK structure $(J,g',\theta')$  on $M$, such that $G\subseteq\mathrm{Aut}_s(M,J,[g'])$.
   \item[$(ii)$] If $G\subseteq\mathrm{Aut}_s(M,J,[g])$, then the action of $G$ lifts to $\widehat  M$.
  \end{enumerate} 
 \end{proposition}
 
 \begin{proof}
$(i)$ Let $\Gamma$ denote the automorphism group of the covering $\pi:\widehat M\to M$. The pull-back $\pi^*\theta$ of the Lee form to $\widehat M$ is exact, so there exists a function $\varphi$ on $\widehat M$ such that $\pi^*\theta=\mathrm{d}\varphi$. Since $\pi^*\theta$ is $\Gamma$-invariant, there exists a group morphism $\Gamma\to (\mathbb{R},+)$, $\gamma\mapsto c_\gamma$, such that $\gamma^*\varphi=\varphi+c_\gamma$ for every $\gamma\in \Gamma$, similar to \eqref{rho}. The K\"ahler metric $g^K:=e^{-\varphi}\pi^* g$ is then $\Gamma$-equivariant, in the sense that $\gamma^*g^K=e^{-c_\gamma}g^K$, for every $\gamma\in \Gamma$.

We claim that $G$ and $\Gamma$ commute. By assumption, for every $a\in G$ and $\hat x\in \widehat M$, we have $\pi(a\hat x)=a\pi(\hat x)$. Consequently, for every $\gamma\in \Gamma$ we obtain 
$$\pi(a\gamma\hat x)=a\pi(\gamma\hat x)=a\pi(\hat x)=\pi(a\hat x).$$
Since $G$ is connected and $\Gamma$ is discrete, this shows that for every $\gamma\in \Gamma$ there exists $\gamma'\in \Gamma$ such that $a\gamma=\gamma' a$ for every $a\in G$. Taking $a=\mathrm{id}$ shows that $\gamma=\gamma'$, thus proving our claim.

Let now $\di\mu$ denote the Haar measure of $G$ (normalized such that $G$ has unit volume) and let $\hat\varphi$ and $\widehat{g^K}$ be the average on $\widehat M$ over $G$ of $\varphi$ and $g^K$ respectively: 
 $$\hat\varphi=\int_Ga^*\varphi\,\di\mu(a),\qquad \widehat{g^K}=\int_Ga^*g^K\,\di\mu(a).$$
Clearly $\widehat{g^K}$ is still K\"ahler, and by construction, taking into account the commutation of $G$ and $\Gamma$, we see that the function $\hat\varphi$ and the metric $\widehat{g^K}$ are $G$-invariant and $\Gamma$-equivariant:
$$\gamma^*\hat\varphi=\hat\varphi+c_\gamma,\qquad \gamma^*\widehat{g^K}=e^{-c_\gamma}\widehat{g^K},\qquad\forall\gamma\in\Gamma.$$
It follows that the metric $e^{\hat\varphi}\widehat{g^K}$ projects 
onto a $G$-invariant lcK metric $g'$ on $M$, whose corresponding Lee form $\theta'$ satisfies $\theta'(X)=0$, for any fundamental vector field $X$ of the $G$-action. Since such an  $X$ is also a Killing vector field with respect to $g'$, we obtain by \eqref{lieg} that $G\subseteq\mathrm{Aut}_s(M,J,[g'])$.

$(ii)$ Any compact Lie group is a finite quotient of $T^k\times G_s$, for some $k\in\mathbb{N}$ and
   $G_s$ a simply connected compact Lie group. By Lemma \ref{lift}, the action of $G_s$ lifts to $\widehat M$, so using Corollary~\ref{product} and Lemma \ref{lift2}, it suffices to show that if $T^k\subseteq\mathrm{Aut}_s(M,J,[g])$ on a compact lcK manifold $(M^{2n},J, [g])$, then its action lifts to the minimal covering $\widehat M$. By Corollary \ref{product} again, we may assume $k=1$. Equivalently, we need to show that the lift of every $X\in\mathfrak{aut}_s(M,J,[g])$ with closed orbits on $M$ has closed orbits on $\widehat M$.

By Proposition \ref{infisom}, $X$ is Killing with respect to the Gauduchon metric $g_0$ on $M$, and its lift $\widehat X$ to $\widehat M$ is Killing with respect to both $\pi^*g_0$ and $g^K$. Since the flow of $X$ is equal to the identity at some time $t_0>0$, it follows that the flow of $\widehat X$ at time $t_0$ is an automorphism of the covering $\widehat M$, and at the same time an isometry of $g^K$, as $\widehat X$ preserves $g^K$. On the other hand, by definition of the minimal covering, the only automorphism of $\widehat M$ which is an isometry of $g^K$ is the identity, thus showing that $\widehat X$ has closed orbits.
 \end{proof}

A straightforward consequence of Proposition~\ref{prop lcklift} is the following:
 
 \begin{corollary}\label{t2lift}
Let $(M^{2n},J)$ be a compact complex manifold and let $G$ be a compact Lie group acting holomorphically on $M$.
  If the $G$-action does not lift to any of the non-compact coverings of $M$, 
  then there exists no lcK structure $(J,[g])$ on $M$, such that $G\subseteq\mathrm{Aut}_s(M,J,[g])$.
 \end{corollary}

\section{Toric Vaisman manifolds}\label{toricVaisman}

In this section, we investigate  toric compact Vaisman manifolds and show that each such manifold is obtained as the mapping torus of a compact toric Sasakian manifold.

Without loss of generality, we assume that the norm of the Lee form of a Vaisman structure equals $1$. Recall that on a Vaisman manifold the Lee and anti-Lee vector fields $\theta^\sharp$ and $J\theta^\sharp$ are Killing and holomorphic. We will need in the sequel the following elementary observation:

\begin{lemma}\label{lem harm}
For any Killing vector field $X$ and any harmonic $1$-form $\alpha$ on a compact connected Riemannian manifold, the following identity holds
$\cL_X\alpha=\di (\alpha(X))=0$. In particular, $\alpha(X)$ is constant. 
\end{lemma}
\begin{proof}
Since the manifold is compact, $\alpha$ is closed and co-closed.
Cartan's formula then yields $\cL_X\alpha=\di (\alpha(X))$.  On the other hand, the $1$-form $\cL_X\alpha$ is harmonic, as $X$ is a Killing vector field. By the Hodge decomposition theorem on a compact Riemannian manifold, the only harmonic and exact differential form is the trivial one. Thus, $\cL_X\alpha=\di (\alpha(X))=0$.
\end{proof}

\begin{lemma}\label{lem vais}
Let $(M^{2n},J,g,\theta)$ be a compact Vaisman manifold of complex dimension $n\ge 2$. 
\begin{enumerate}[label=$(\roman*)$]
\item If $\alpha\in\Omega^1(M)$ is a harmonic form pointwise orthogonal on $\theta$, then $J\alpha$ is also harmonic and pointwise orthogonal on $\theta$.
\item The following decomposition of the space of harmonic $1$-forms holds:
\begin{equation}\label{dec}
\mathcal{H}^1(M)=\mathrm{span}\{\theta\}\oplus \mathcal{H}^1_0(M),
\end{equation}
where each element of $\mathcal{H}^1_0(M)$ is pointwise orthogonal on $\theta$ and $\mathcal{H}^1_0(M)$
is $J$-invariant.
\end{enumerate}
\end{lemma}

\begin{proof}
\begin{enumerate}[label=$(\roman*)$, leftmargin=0.5cm]
\item  Lemma \ref{lem harm} applied to the harmonic form $\alpha$ and the Killing vector fields $\theta^\sharp$ and $J\theta^\sharp$ yields
\begin{equation}\label{liedi}
 \mathcal{L}_{\theta^\sharp}\alpha=0, \quad\mathcal{L}_{J\theta^\sharp}\alpha=
\di(\alpha(J\theta^\sharp))=0
\end{equation}
and the function $\<J\theta,\alpha\>=:c$ is constant on $M$.
Furthermore, by Cartan's formula and using the fact that $\theta^\sharp$ is a holomorphic vector field, we obtain
\begin{equation}\label{eq tja}
\theta^\sharp\lrcorner \di J\alpha=\di\<J\alpha,\theta\>+\theta^\sharp\lrcorner \di J\alpha=\mathcal{L}_{\theta^\sharp}(J\alpha)=0.
\end{equation}

Recall that on the universal covering $\witi M$, endowed with the pull-back lcK structure $(\witi J, \ti g, \ti \theta )$, the metric $\ti g$ and the K\"ahler metric $g^K$ are conformally related by $g^K=e^{-\varphi}\ti g$, with $\di\varphi=\ti \theta$. Hence, 
their corresponding Laplace operators acting on a $1$-form $\beta$ on $\witi M$ are related as follows (see for instance \cite[Theorem~1.159]{Besse2008}):
\begin{equation}\label{laplaceconf}
\Delta^K\beta=e^\varphi\left(\witi \Delta \beta+(n-1)\di (\<\beta, \ti\theta\>)+(n-2)\ti\theta^\sharp\lrcorner\di\beta+(n-1)\<\beta, \ti\theta\>\ti\theta+(\ti\delta\beta)\ti\theta\right),
\end{equation}
where all operators of the right hand side are associated to the metric $\ti g$.
Applying this formula for $\beta=\ti\alpha$, the pull-back of $\alpha$, and taking into account that $\ti\alpha$ is closed and co-closed and orthogonal on $\ti\theta$, we obtain that $\ti\alpha$ is also harmonic with respect to the K\"ahler metric, \emph{i.e.} $\Delta^K\ti\alpha=0$. Since on a K\"ahler manifold the Laplace operator commutes with the complex structure, it follows that $\Delta^K\witi J\ti\alpha=\witi J\Delta^K\ti\alpha=0$. 
Applying now~\eqref{laplaceconf} to $\beta=\witi J\ti\alpha$ and projecting on $M$, we obtain:
\begin{equation} \label{Jth}
 \Delta  J\alpha+(n-1)\di (\< J\alpha,  \theta\>)+(n-2) \theta^\sharp\lrcorner\di  J\alpha+(n-1)\< J\alpha,  \theta\> \theta+ (\delta  J\alpha)\theta=0,
\end{equation}
which further simplifies, using \eqref{eq tja} and the fact that $\< J\theta,\alpha\>=c$ is constant, to:
\begin{equation} \label{Jth1}
\Delta J \alpha+(n-1)c\,\theta+(\delta J\alpha)\theta=0.
\end{equation}
Taking in~\eqref{Jth1} the scalar product with $\theta$ and integrating over $M$ yields $(n-1)c=0$. Since $n\ge2$, we get $0=c=\<J\theta,\alpha\>$, so $J\alpha$ is pointwise orthogonal to $\theta$. Applying the codifferential to~\eqref{Jth1} yields
\[\delta\di\delta J\alpha-\theta^\sharp( \delta J\alpha)=0.\]
On the other hand, using the fact that $\theta^\sharp$ is Killing and holomorphic, together with \eqref{liedi}, we obtain $\theta^\sharp( \delta J\alpha)= \mathcal{L}_{\theta^\sharp}\delta J\alpha=0$. Hence, we get
\[\Delta\delta J\alpha=\delta\di\delta J\alpha=0.\]
Thus, the function $\delta J\alpha$ is constant, and as its integral over $M$ vanishes, we get that $\delta J\alpha$ is identically zero.  Substituting this into~\eqref{Jth1} implies that $J\alpha$ is harmonic. 
\item Let $\beta\in\mathcal{H}^1(M)$. By Lemma~\ref{lem harm}, the function $\<\beta, \theta\>$ is constant and thus $\beta-\<\beta, \theta\>\theta$ is harmonic. According to $(i)$, this completes the proof.
\end{enumerate}
\end{proof}

\begin{remark}
A direct consequence of Lemma~\ref{lem vais} is the fact that the first Betti number of a compact Vaisman manifold is odd, which was proven in~\cite{ks}.
\end{remark}

\begin{proposition}\label{toricvais}
Let $(M^{2n},J,g)$ be a compact toric Vaisman manifold. Then the following assertions hold:
\begin{enumerate}[label=$(\roman*)$]
\item The lcK rank of $(J,[g])$ equals $1$.
\item $M$ is obtained as the mapping torus of a compact toric Sasakian manifold.
\end{enumerate}
\end{proposition}
\begin{proof}
$(i)$ The rank of any lcK structure on $M$ is less than or equal to the first Betti number of $M$. Hence, it is enough to show that $\mathrm{H}_1(M, \mR)\cong \mR$. According to the Hodge theory on compact manifolds and to the decomposition \eqref{dec}, it suffices to show that $\mathcal{H}^1_0(M)=\{0\}$. 

Let $\beta\in\mathcal{H}^1_0(M)$ and let $\xi$ be a fundamental vector field of the toric action. By Lemma~\ref{lem harm} applied to $\beta$ and to the Killing vector field $\xi$, the function $\beta(\xi)$ is constant. 
On the other hand, $\xi$ is twisted Hamiltonian, so there exists $h\in\mathcal{C}^\infty(M)$, such that $(J\xi)^\flat=\xi\lrcorner\Omega=\di^\theta h=\di h-h\theta$. We now integrate over $M$ the constant $\beta(\xi)$:
\[\int_M \beta(\xi) \di v=\int_M \<J\beta, (J\xi)^\flat\> \di v=\int_M  \<J\beta, \di h-h\theta\> \di v=\int_M (\delta( J\beta) h-h \<J\beta,\theta\>)\di v=0,\]
where for the last equality we used that $J\beta\in \mathcal{H}^1_0(M)$ {\em cf.} Lemma~\ref{lem vais}, so $J\beta$ is orthogonal to $\theta$ and co-closed. Hence, $\beta(\xi)=0$. In particular, the same holds also for $J\beta$, since $\mathcal{H}^1_0(M)$ is $J$-invariant, by {Lemma~\ref{lem vais}}.

Let us now consider a basis of the Lie algebra of the torus $T^n$, $\{\xi_1, \ldots,\xi_n\}$. We denote by the same symbols the corresponding fundamental vector fields on $M$. According to the principal orbit theorem, there exists a dense open set $M_0$ of $M$, on which $\{\xi_1, \ldots,\xi_n, J\xi_1, \ldots,J\xi_n\}$ forms a basis of the tangent bundle of $M_0$. As shown above, each $\beta\in\mathcal{H}^1_0(M)$ vanishes on $\{\xi_1, \ldots,\xi_n, J\xi_1, \ldots,J\xi_n\}$, hence $\beta=0$ on $M_0$, and by density also on $M$. We thus conclude that $\mathcal{H}^1_0(M)=\{0\}$.

$(ii)$ We recall that the minimal covering $\widehat M$ of a Vaisman manifold $M$, endowed with the K\"ahler metric $g^K$, is biholomorphic and isometric to the K\"ahler cone over a Sasakian manifold $S$. Moreover, each homothety of $(\mR\times S,g^K)$ is of the form $(t,p)\mapsto (t+a, \psi_a(p))$, for some $a\in\mR$ and $\psi_a$ an isometry of $S$. If the Vaisman structure is assumed to be toric, then by \cite[Theorem~4.9]{p}, the Sasakian manifold $S$ is also toric.

Since by $(i)$, the lcK rank of the toric Vaisman  structure equals $1$, the deck transformation group of the minimal covering $\widehat M$ is $\Gamma\cong\mZ$. Using this identification, we assume that $\Gamma$ is generated by the strict homothety $(t,p)\mapsto (t+1, \psi_1(p))$. Hence, $M$ is  obtained as the mapping torus $[0,1]\times S/\sim$, where $(0,p)\sim (1,\psi_1(p))$. Moreover, this also shows that $S$ is compact.
\end{proof}

\begin{remark}
According to \cite[Theorem~1.1]{lerman04}, the fundamental group of a compact toric Sasakian manifold is a finite abelian group. Hence, from the proof of Proposition~\ref{toricvais}, $(ii)$, it follows that the universal covering of a compact toric Vaisman manifold is the K\"ahler cone over a compact toric Sasakian manifold. We recall  that compact toric contact manifolds were classified by E.~Lerman in \cite[Theorem~2.18]{lerman}. 
\end{remark}

\section{The Kodaira dimension of toric lcK manifolds}\label{seckod}

We denote by $\kappa(M)$ the Kodaira dimension of a compact complex manifold $(M, J)$. For the definition of the Kodaira dimension and its properties, as well as for the Kodaira-Enriques classification of compact complex surfaces we refer to \cite{bpv} or \cite{k1}, \cite{k2}.

\begin{theorem}\label{kod}
The Kodaira dimension of a compact toric strict lcK manifold is $-\infty$.
\end{theorem}

\begin{proof}
Let $(M^{2n},J,[g])$ be a compact toric strict lcK manifold.
An effective action of the torus $T^n$ has $n$-dimensional principal orbits, and their union is a dense open set $M_0$ in $M$ by the principal orbit theorem (see \emph{e.g.} \cite[Theorem~2.8.5]{dk}). Consequently, there exist real holomorphic vector fields $X_1,\ldots,X_n$ which mutually commute and such that $X_1,\ldots,X_n$ are linearly independent on $M_0$. Since $X_1,\ldots,X_n$ are twisted Hamiltonian, Lemma~\ref{inclusion} and Proposition~\ref{isotropic} show that $\Omega(X_j,X_k)=0$ for every $j,k\le n$. Consequently $X_1,\ldots,X_n,JX_1,\ldots,JX_n$ are linearly independent on $M_0$. Moreover, these $2n$ vector fields mutually commute, since $X_1,\ldots,X_n$ are holomorphic and mutually commute.
We thus obtain
holomorphic sections $Z_j:=X_j-iJX_j$ of $T^{1,0}M$  satisfying $[Z_j,Z_k]=[Z_j,\bar Z_k]=0$ for every $j,k\in\{1,\ldots,n\}$ and such that $Z_1,\ldots,Z_n$ are linearly independent on $M_0$. The anti-canonical bundle $(K_M)^*\simeq \Lambda^n(T^{1,0}M)$ thus has a non-trivial holomorphic section $\sigma:=Z_1\wedge\ldots \wedge Z_n$. 

Assume, for a contradiction, that some positive power $(K_M)^{\otimes k}$ of the canonical bundle has a non-trivial holomorphic section $\alpha$. Then $\alpha(\sigma^{\otimes k})$ is a non-trivial holomorphic function on $M$, thus it is a non-zero constant. Consequently, $\sigma$ is nowhere vanishing, so $\{Z_j\}_{1\leq j\leq n}$ is a  basis of $T^{1,0}M$ at each point of $M$. We denote by $\omega_j\in\Omega^{1,0} M$ the dual basis and claim that $\omega_j$ are closed holomorphic 1-forms. Indeed, since $\omega_i(Z_j)=\delta_{ij}$ and $\omega_i(\bar Z_j)=0$, we obtain for all $i,j,k$:
$$\di\omega_i(Z_j,Z_k)=Z_j(\omega_i(Z_k))-Z_k(\omega_i(Z_j))-\omega_i([Z_j,Z_k])=0,$$
$$\di\omega_i(Z_j,\bar Z_k)=Z_j(\omega_i(\bar Z_k))-\bar Z_k(\omega_i(Z_j))-\omega_i([Z_j,\bar Z_k])=0.$$
Define now the real $(1,1)$-form
$\Omega:=i\sum_{j=1}^n \omega_j\wedge \ol\omega_j$. From the above we have that $\Omega$ is closed. Moreover, the symmetric bilinear form $\Omega(\cdot,J\cdot)$ is positive definite, so $\Omega$ is a K\"ahler form on $M$ compatible with $J$.
By a result of I.~Vaisman, \cite[Thm. 2.1]{v},  the lcK structure is then globally conformally K\"ahler, contradicting the assumption that the lcK structure is strict.

This contradiction shows that the positive tensor powers of $K_M$ have no holomorphic sections, hence $\kappa(M)=-\infty$. 
\end{proof}

\section{Toric lcK surfaces}
 
Using the Kodaira classification of compact complex non-K\"ahlerian surfaces and Theorem~\ref{kod}, we will now describe all compact complex surfaces carrying a toric strict lcK structure. Let us first recall the definition of Hopf surfaces.

\begin{definition}\label{def hopf}
A primary Hopf surface is a complex surface  $(\mC^2\setminus \{0\})/\Gamma$, where 
$\Gamma$ is the group generated by
$$\gamma(z_1,z_2):=( \beta z_1,\alpha z_2+\lambda z_1^m),$$
with $\alpha,\beta,\lambda \in\mC$, satisfying $0<|\alpha|,|\beta|<1$, $\lambda(\alpha-\beta^m)=0$ and $m\in\mN\setminus\{0\}$. The following two cases are distinguished: 
\begin{enumerate}[label=$(\roman*)$]
\item The  diagonal primary Hopf surfaces, when $\lambda=0$.
\item The  non-diagonal primary Hopf surfaces, when $\lambda\neq 0$ (and thus $\alpha=\beta^m$).
\end{enumerate} 
A secondary Hopf surface is a finite quotient of a non-diagonal primary Hopf surface.
\end{definition}

We will need in the sequel the following standard result which relates holomorphic vector fields on a complex manifold and on its blow-up at some point (see for instance \cite{chk}):
\begin{lemma}\label{blowup} 
 Let $M$ be a compact complex manifold. The holomorphic vector fields on the blow-up of $M$ at a point $p$
 are exactly the lifts of holomorphic vector fields on $M$ that vanish at $p$.
\end{lemma}

The main result of this section is the following:

\begin{theorem}\label{mainthm}
The only compact complex surfaces admitting a toric strict lcK structure are the diagonal Hopf surfaces.
\end{theorem}

\begin{proof}
Let $(M, J)$ be a compact complex surface carrying a toric strict lcK structure. By Theorem~\ref{kod}, $\kappa(M)=-\infty$. The Kodaira classification of non-K\"ahlerian
 compact complex surfaces shows that the only possible examples with Kodaira dimension $-\infty$ are the Inoue surfaces, the Hopf surfaces, the surfaces in the class $VII_0^+$ ({\em i.e.} with $b_2>0$), and blow-ups thereof (see \emph{e.g.} \cite{k1} or \cite[Proposition 1]{b}).
 
The space of holomorphic vector fields is at most one-dimensional on Inoue surfaces (see \cite[Prop. 2 $(ii)$, Prop. 3 $(ii)$ and Prop. 5]{i} or \cite[Prop. 12]{b}) and on surfaces in the class $VII_0^+$ (see {\em e.g.} \cite[Remark 0.2]{dot}). Consequently, these surfaces cannot admit an effective holomorphic $T^2$-action. By Lemma~\ref{blowup}, the same holds for each of their blow-ups.
 
We now consider the two types of Hopf surfaces.

\noindent{\bf 1) Non-diagonal Hopf surfaces and blow-ups thereof.} 
We claim that these surfaces do not admit a toric lcK structure.
  
Since the secondary Hopf surfaces are finite quotients of non-diagonal primary Hopf surfaces, it suffices to show that the latter ones do not admit such toric lcK structures. Let $M:=(\mC^2\setminus\{0\})/\Gamma$ be a non-diagonal primary Hopf surface, where $\Gamma$ is generated by $\gamma(z_1,z_2):=( \beta z_1,\beta^m z_2+\lambda z_1^m)$, according to Definition \ref{def hopf}. We first determine the space of holomorphic vector fields on $M$. These correspond to the holomorphic vector field on $\mC^2\setminus\{0\}$ invariant under the action of $\Gamma$.
  
 \noindent  {\bf Claim.} The space of $\Gamma$-invariant holomorphic vector fields on $\mC^2\setminus\{0\}$ is generated by 
$$X_1:=z_1\frac{\partial}{\partial z_1}+m z_2\frac{\partial}{\partial z_2},\quad X_2:=z_1^m\frac{\partial}{\partial z_2}.$$

\noindent {\it Proof of the Claim.} Let $X=u\frac{\partial}{\partial z_1}+v\frac{\partial}{\partial z_2}$ be a $\Gamma$-invariant holomorphic vector field on $\mC^2\setminus\{0\}$, where $u$ and $v$ are two holomorphic functions on $\mC^2\setminus\{0\}$. Since the singularity is isolated, $u$ and $v$ extend holomorphically to the whole complex plane $\mC^2$. The $\Gamma$-invariance of $X$ reads
\begin{gather}
u\circ \gamma (z)=\beta u(z),\label{eq u}\\
v\circ \gamma(z)=m\lambda z_1^{m-1}u(z)+\beta^mv(z),\label{eq v}
\end{gather}
for all $z\in\mC^2$. On the other hand, by induction on $n\in\mN$, we establish that  
\begin{equation}\label{gamman}\gamma^n(z)=(\beta^n z_1,\beta^{mn}z_2+n\lambda \beta^{m(n-1)}z_1^m).\end{equation}

In the sequel, we use the fact that for every $\beta\in\mC$ with $0<|\beta|<1$, $k\in\mN$ and $n\in\mN\setminus\{0\}$, an entire function $f$, which satisfies $f(\beta^n z)=\beta^kf(z)$ for any $z\in\mC$, is either a monomial of degree $\frac{k}{n}$, if $n$ divides $k$, or identically zero otherwise.

For $z_1=0$, \eqref{eq u} reads $u(0,\beta^mz_2)=\beta u(0,z_2)$. By the above remark, the function $z_2\mapsto u(0,z_2)$ is linear if $m=1$, or identically zero otherwise. 
Hence, there exists $c_m\in\mC$ (which vanishes if $m\ne 1$), such that $z_1=0$  is a zero of the function $z_1\mapsto u(z_1,z_2)-c_mz_2$. Thus  $\ti u(z):=\frac{u(z)-c_mz_2}{z_1}$ is an entire function and by \eqref{eq u} it satisfies $\ti u\circ \gamma=\ti u-\frac{c_m\lambda}{\beta}$. Therefore, $\ti u\circ \gamma^n(z)=\ti u(z)-n\frac{c_m\lambda}{\beta}$ for every $z\in\mC^2\setminus\{0\}$. Taking the limit in this identity for $n\to\infty$, and using the fact that $\gamma^n(z)$ tends to $0$ for every $z$ by \eqref{gamman}, it follows that $c_m=0$ and $\ti u$ is constant. We deduce that $u(z)=cz_1$ for some constant $c\in\mC$. Substituting in \eqref{eq v} and iterating, we obtain by immediate induction
\begin{equation} \label{eq vn}
v\circ\gamma^n(z)=\beta^{nm}v(z)+nmc\lambda\beta^{m(n-1)}z_1^m, \quad \forall n\in\mN.
\end{equation}
In particular, for $n=1$ and $z_1=0$, we readily obtain $v(0,\beta^m z_2)=\beta^{m}v(0,z_2)$, for all $z_2\in\mC$.  So, there exists $c'\in\mC$, such that $v(0,z_2)=c'z_2$, hence the entire function $z_1\mapsto v(z_1,z_2)-c'z_2$ vanishes at $z_1=0$. Therefore, either $v(z)=c'z_2$, for any $z\in\mC^2$, or there exists $k\in\mN\setminus\{0\}$ and an entire function $\ti v$, such that $v(z)-c'z_2=z_1^k\ti v(z)$ and $\ti v(0,a)\neq 0$, for some $a\in\mC$. In the first case, \eqref{eq v} implies that $c'=mc$. We now consider the second case.
By~\eqref{eq vn}, the function $\ti v$ satisfies
$$\ti v\circ \gamma^n(z) =\beta^{(m-k)n}(\ti v+n\lambda\beta^{-m}(mc-c')z_1^{m-k}).$$

We now let $n\to\infty$ and distinguish the following cases:
\begin{itemize}
	\item if $k>m$, then $\ti v\equiv 0$ and $c'=mc$.
	\item if $k<m$, then the function $\ti v$ vanishes at $(0,z_2)$, for all $z_2\in\mC$. This contradicts the definition of $k$ and $\ti v$.
	\item if $k=m$, then $c'=mc$ and $\ti v$ is constant.
\end{itemize}
We conclude that $u(z)=cz_1$ and  $v(z)=mcz_2+c''z_1^m$, for some constants $c,c''\in\mC$.
This proves the claim.  \hfill $\checkmark$

Since the vector fields $X_1$ and $X_2$ commute, the flow of any holomorphic vector field $X:=aX_1+bX_2$, with $a,b\in\mathbb{C}$, is the composition of the flows of $aX_1$ and $bX_2$, which are given by
  \[\varphi_t(z_1, z_2)=(e^{at}z_1, e^{mat}z_2), \quad \psi_t(z_1, z_2)=(z_1, z_2+btz_1^m).\]
Namely, the flow of $X$ is $(\varphi_t\circ\psi_t)(z_1,z_2)=(e^{at}z_1,e^{mat}(z_2+btz_1^m))$.
  Hence, the orbits of $X$ are all relatively compact if and only if $a\in i\mathbb{R}$ and $b=0$,
  \emph{i.e.} $X$ is a multiple of $X_1$. On the other hand, as the fundamental group of $M$ is $\mathbb{Z}$, the minimal covering of any lcK structure 
 on $M$ is equal to its universal covering $\widetilde M=\mathbb{C}^2\setminus\{0\}$.

If $(M,J)$ admits a toric lcK metric $g$, then Lemma~\ref{inclusion} implies that
$T^2\subseteq\mathrm{Aut}_s(M,J,[g])$.
Furthermore, by Proposition~\ref{prop lcklift}, 
the $T^2$-action lifts to $\mathbb{C}^2\setminus\{0\}$. This contradicts the above computation, which shows
  that there do not exist two linearly independent holomorphic $\Gamma$-invariant vector fields on $\mathbb{C}^2\setminus\{0\}$ with relatively compact orbits. We conclude that a non-diagonal primary Hopf surface does not carry 
a toric lcK structure. This argument also shows that the same is true for any of its blow-ups.

Note that any non-diagonal primary Hopf surface $(M,J)$ admits an lcK metric (\emph{cf.} \cite[Proposition 11]{b}), which can be averaged to an lcK metric $g$ by an argument of L.~Ornea and M.~Verbitsky, \cite{ov}, such that $T^2\subseteq\mathrm{Aut}(M,J,[g])$. However, as seen above, $T^2$ is never a subgroup of $\mathrm{Aut}_s(M,J,[g])$
  
\noindent {\bf 2) Diagonal Hopf surfaces.}  Let $M_{\alpha,\beta}:=(\mC^2\setminus\{0\})/\Gamma$ be a diagonal Hopf surface, where $\Gamma$ is generated by $\gamma(z_1,z_2):=( \beta z_1,\alpha z_2)$, according to Definition \ref{def hopf}.

A Vaisman metric on $M_{\alpha,\beta}$ was constructed by P.~Gauduchon and L.~Ornea \cite{go}, as follows: let $\phi_{\alpha,\beta}:\mC^2\setminus\{0\}\longrightarrow \mR$ be the function 
implicitly defined as the unique solution of the  equation
$$|z_1|^2|\beta|^{-2\phi_{\alpha,\beta}(z)}+|z_2|^2|\alpha|^{-2\phi_{\alpha,\beta}(z)}=1,$$
for any $z\in\mC^2\setminus\{0\}$. It can be checked that the function $|\alpha\beta|^{\phi_{\alpha,\beta}}$ is a potential for a K\"ahler metric $g^K:=\di\di^c|\alpha\beta|^{\phi_{\alpha,\beta}}$ on $\mC^2\setminus\{0\}$ and $g:=|\alpha\beta|^{-\phi_{\alpha,\beta}}g^K$ projects to a Vaisman metric on $M_{\alpha,\beta}$ (for details see \cite[Proposition 1]{go}). 

\begin{lemma}\label{rem}
The above defined Vaisman structure on 
$M_{\alpha,\beta}$ is toric. 
\end{lemma}
\begin{proof}
We first notice that $\mC^2\setminus\{0\}$ carries a $T^2$-action given by 
$$T^2\longrightarrow {\rm Diff}(\mC^2\setminus\{0\}), \quad (t_1,t_2)\cdot(z_1,z_2)=(t_1 z_1,t_2z_2),$$ which is effective, holomorphic and $\Gamma$-invariant and hence descents to an effective, holomorphic $T^2$-action on $M_{\alpha,\beta}$. 
The potential $\phi_{\alpha,\beta}$ is smooth and satisfies $\phi_{\alpha,\beta}\circ \gamma=\phi_{\alpha,\beta}+1$, for any $\gamma\in\Gamma$, and is constant along the
$T^2$-orbits, so $T^2$  acts isometrically on $\mC^2\setminus\{0\}$  with respect to $g^K$. As the conformal factor between $g$ and $g^K$ is exactly this potential, it follows that $T^2$ also acts isometrically with respect to $g$ and is a subgroup of $\mathrm{Aut}_s(M_{\alpha,\beta})$. By Lemma~\ref{invincl}, the $T^2$-action is twisted Hamiltonian. 
\end{proof}

In order to prove that a blow-up at a point of $M_{\alpha,\beta}$ does not admit an effective holomorphic $T^2$-action, 
we start by determining the space of holomorphic vector fields on $M_{\alpha,\beta}$. Namely, we have the following

\noindent{\bf Claim.}
The space of $\Gamma$-invariant holomorphic vector fields on $\mC^2\setminus\{0\}$ is generated by 
\begin{enumerate}[label=\alph*)]
\item $\bigl\{z_1\frac{\partial}{\partial z_1}, z_2\frac{\partial}{\partial z_2}, z_2\frac{\partial}{\partial z_1}, z_1\frac{\partial}{\partial z_2}\bigr\}$, if $\alpha=\beta$.
 \item $\bigl\{z_1\frac{\partial}{\partial z_1}, z_2\frac{\partial}{\partial z_2}, z^m_2\frac{\partial}{\partial z_1}\bigr\}$, if $\alpha=\beta^m$ for some integer $m\geq 2$.
   \item $\bigl\{z_1\frac{\partial}{\partial z_1}, z_2\frac{\partial}{\partial z_2}, z^m_1\frac{\partial}{\partial z_2}\bigr\}$, if $ \beta=\alpha^m$ for some integer $m\geq 2$. 
  \item $\bigl\{z_1\frac{\partial}{\partial z_1}, z_2\frac{\partial}{\partial z_2}\bigr\}$, if $\alpha\neq\beta^{k}$ and $\beta\neq\alpha^{k}$ for any $ k\in \mN$.
   
  \end{enumerate}

\noindent{\it Proof of the Claim.}
Like before, any holomorphic vector field on $\mC^2\setminus\{0\}$ can be extended to $\mC^2$, since the singularity is isolated. Let $X:=u\frac{\partial}{\partial z_1}+v\frac{\partial}{\partial z_2}$ be a holomorphic vector field,  where $u$ and $v$ are two holomorphic functions on $\mC^2$. The vector field $X$ is $\Gamma$-invariant if and only if $u(\alpha z_1, \beta z_2)=\alpha u(z_1,z_2)$ and $v(\alpha z_1, \beta z_2)=\beta v(z_1,z_2)$. Equivalently, $a_{jk}(\alpha-\alpha^j\beta^k)=0$ and $b_{jk}(\beta-\alpha^j\beta^k)=0$ for all $j,k\in\mN$, where $a_{jk}$ and $b_{jk}$ are the coefficients of the power series of $u$ and $v$ respectively. Since $0<|\alpha|,|\beta|<1$, it turns out that in each of the four possible cases, $X$ is a linear combination of the vector fields as stated in the claim. \hfill $\checkmark$
  
The next step is to show that there are no two linearly independent commuting holomorphic vector fields which vanish at the same point. Note that it suffices to prove this for the cases a) and b) in the above claim. 

\noindent {\bf Case a)} 
Assume that $\alpha=\beta$ and let $X,Y\in {\rm span}\bigl\{z_1\frac{\partial}{\partial z_1}, z_2\frac{\partial}{\partial z_2}, z_2\frac{\partial}{\partial z_1}, z_1\frac{\partial}{\partial z_2}\bigr\}$, such that there exists $w\in \mC^2\setminus\{0\}$ with $X_w=Y_w=0$. Hence, there exist $a\in\mC^2\setminus\{0\}$ and $c_j\in\mC$, such that 
$X_z=(a_1z_1+a_2z_2)\left(c_1\frac{\partial}{\partial z_1}+c_2\frac{\partial}{\partial z_2}\right)$ and $Y_z=(a_1z_1+a_2z_2)\left(c_3\frac{\partial}{\partial z_1}+c_4\frac{\partial}{\partial z_2}\right)$. Therefore, the Lie bracket of $X$ and $Y$ is given by
\begin{equation*}
\begin{split}
[X,Y]_z&=(a_1z_1+a_2z_2)\left( (a_1c_1+a_2c_2)\left(c_3\frac{\partial}{\partial z_1}+c_4\frac{\partial}{\partial z_2}\right)-(a_1c_3+a_2c_4)\left(c_1\frac{\partial}{\partial z_1}+c_2\frac{\partial}{\partial z_2}\right)\right)\\
& =(a_1z_1+a_2z_2)(c_2c_3-c_1c_4)\left(a_2 \frac{\partial}{\partial z_1} -a_1\frac{\partial}{\partial z_2}\right).
\end{split}
\end{equation*}
Thus, if $X$ and $Y$ commute, then $c_2c_3=c_1c_4$, hence $X$ and $Y$ are linearly dependent.
 
 \noindent {\bf Case b)}  Assume that $\alpha=\beta^m$ for some integer $m\geq 2$  and let $X,Y\in {\rm span}\bigl\{z_1\frac{\partial}{\partial z_1}, z_2\frac{\partial}{\partial z_2}, z_2^n\frac{\partial}{\partial z_1}\bigr\}$, such that there exists $w\in \mC^2\setminus\{0\}$ with $X_w=Y_w=0$. This readily shows that there exist $a,b\in\mC^2\setminus\{0\}$ such that either
 \begin{gather*}
 X_z=a_1z_2^m\frac{\partial}{\partial z_1}+a_2z_2\frac{\partial}{\partial z_2}, \quad Y_z=b_1z_2^m\frac{\partial}{\partial z_1}+b_2z_2\frac{\partial}{\partial z_2} \quad {\rm or}\\
 X_z=(a_1z_1+a_2z_2^m)\frac{\partial}{\partial z_1}, \quad Y_z=(b_1z_1+b_2z_2^m)\frac{\partial}{\partial z_1}.
 \end{gather*}
Therefore, we compute
$$[X,Y]_z=m(a_2b_1-a_1b_2)z_2^m\frac{\partial}{\partial z_1}, \textrm{ respectively }[X,Y]_z=b_1X_z-a_1Y_z.$$
In both cases, $X$ and $Y$ commute if and only if they are linearly dependent.

By applying Lemma~\ref{blowup}, we conclude that no blow-up of $M_{\alpha,\beta}$ carries two linearly independent commuting holomorphic vector fields, and in particular cannot be toric lcK.
\end{proof}

As an immediate consequence of Theorem~\ref{mainthm} and Lemma~\ref{rem}, we obtain:

\begin{corollary} 
If $(M,J)$ is a compact complex surface admitting a toric lcK structure, then it also admits a toric Vaisman structure.
\end{corollary}

It would be interesting to know whether this result holds for all compact complex manifolds. However, the methods of this paper, which are based on the Kodaira classification of compact complex surfaces, do not extend to higher dimensions.

\end{document}